\documentclass[12pt]{amsart}
\usepackage{graphics,wrapfig, graphicx, mathrsfs,xcolor}
\usepackage{mathtools}
\usepackage{amssymb,amscd,graphicx,xcolor,subfig,tikz}
\usepackage{graphics}
\usepackage{indentfirst}
\usepackage{bm, enumerate,xcolor}
\usepackage[dvips]{epsfig}
\usepackage{latexsym}
\usepackage{float}

\usepackage{graphics,url,color,epsfig}
\usepackage{tikz}
\usepackage{xcolor}
\definecolor{mysoftblue}{hsb}{0.55,0.15,0.9}
\definecolor{mysoftpurple}{hsb}{0.9,0.25,0.9}
\definecolor{mysoftorange}{hsb}{0.15,0.3,0.95}
\usetikzlibrary{patterns}

\usepackage[colorlinks]{hyperref}
\AtBeginDocument{
   \hypersetup{
    linkcolor=blue,
    citecolor=blue,
 }
}

\usepackage{amsmath}
\usepackage[applemac]{inputenc}
\usepackage[T1]{fontenc}
\newtheorem{lemma}{Lemma}[section]
\newtheorem{theorem}{Theorem}[section]
\newtheorem{corollary}{Corollary}[section]

\newtheorem{definition}{Definition}[section]
\newtheorem{remark}{Remark}[section]

\numberwithin{equation}{section} \numberwithin{theorem}{section}
\numberwithin{example}{section} \numberwithin{remark}{section}
\numberwithin{figure}{section} \numberwithin{algorithm}{section}
\mathtoolsset{showonlyrefs}

\begin{document}
\title[Phase transitions in two-component Bose-Einstein condensates (I)]{Phase transitions in two-component Bose-Einstein condensates with Rabi frequency (I): The De Giorgi conjecture for the local problem in $\mathbb{R}^{3}$}
\author{Leyun Wu}
\address{School of Mathematics, South China University of Technology, Guangzhou, 510640, P. R. China}\email{leyunwu@scut.edu.cn}
\author{Chilin Zhang}
\address{School of Mathematical Sciences, Fudan University, Shanghai 200433, P. R. China}\email{zhangchilin@fudan.edu.cn}

\begin{abstract}

We consider the following elliptic system modeling Rabi-coupled two-component 
Bose-Einstein condensates,
\[
\begin{cases}
\Delta u = u(u^{2}+v^{2}-1) + v(\alpha uv-\omega),\\[1mm]
\Delta v = v(u^{2}+v^{2}-1) + u(\alpha uv-\omega),
\end{cases}
\qquad u,v>0 \ \text{in }\mathbb{R}^{3},
\]
and establish the first De~Giorgi-type classification result for this competitive 
system. We prove that any entire solution satisfying the monotonicity condition
\[
\partial_{x_{3}}u>0>\partial_{x_{3}}v
\]
is necessarily one-dimensional and connects the two stable equilibria determined 
by the parameters $(\alpha,\omega)$.

The proof relies on two new ingredients. 
First, a sharp range estimate reveals a strictly negative defect structure in the 
linearized operator, leading to a divergence-form identity that forces the directional 
slopes of all level sets to be constant. 
Second, we develop a translation-invariant energy evolution method for the associated 
coupled Ginzburg-Landau functional, yielding an $O(R^{2})$ energy-growth bound and a 
Liouville-type theorem.

Our results extend the classical classification results from the scalar Allen-Cahn equation to a physically relevant coupled system, and provide the first rigorous confirmation that De Giorgi's phenomenon persists under Rabi coupling.


\end{abstract}

\maketitle
\\
2020 {\bf MSC.} 35Q56, 82B26, 35J47, 35B53.

\section{Introduction}
In the 1970s, De Giorgi (\cite{DeGiorgi1979}) proposed a celebrated conjecture  concerning the
symmetry of monotone solutions to semilinear elliptic equations. More precisely, let $u$
be a bounded entire solution of the Allen-Cahn equation
\begin{equation}\label{eq:AC}
    \Delta u = u^{3}-u \quad \text{in } \mathbb{R}^{n}
\end{equation}
satisfying the monotonicity condition
\[
    \frac{\partial u}{\partial x_{n}} > 0 \quad \text{in } \mathbb{R}^{n}.
\]
De Giorgi conjectured that, at least for dimensions $n \leq 8$, the level sets
of $u$ must be hyperplanes. Equivalently, any such monotone solution should be
one-dimensional, i.e., $u(x) = g(x \cdot \nu)$ for some function $g:\mathbb{R}
\to \mathbb{R}$ and unit vector $\nu \in \mathbb{R}^{n}$.

This conjecture lies at the intersection of elliptic PDE, phase transition theory, and geometric analysis, since the level sets behave asymptotically like minimal surfaces, connecting the conjecture to the classical Bernstein problem. Over the years, it has inspired a vast literature and has become fundamental to the qualitative analysis of nonlinear elliptic equations，  providing a natural bridge between
elliptic PDEs, minimal surface theory, and problems  in mathematical
physics.

\subsection{The system of Bose-Einstein condensates}

The present paper extends the De~Giorgi framework from single equations to a 
competitive elliptic system arising in two-component Bose-Einstein condensates (BECs). 
We consider positive solutions $(u,v)$ of
\begin{equation}\label{eq. main}
    \left\{\begin{aligned}
        \Delta u=u(u^{2}+v^{2}-1)+v(\alpha uv-\omega),\\
        \Delta v=v(u^{2}+v^{2}-1)+u(\alpha uv-\omega),
    \end{aligned}
    \right.\quad u,v>0.
\end{equation}

The system arises naturally in the physics of two-component Bose-Einstein condensates (BECs) exhibiting partial phase transition. It is derived from the Gross-Pitaevskii energy functional describing two-component condensates with both intra- and interspecies interactions. In contrast to the classical two-component segregated BEC models, the present system couples the two equations both linearly and nonlinearly, due to the spin coupling of the hyperfine states in addition to the intercomponent interaction. Here $u$ and $v$ represent the wave functions of the two components, $\alpha>0$ denotes the interaction parameter, and $\omega$ is the Rabi frequency corresponding to the one-body coupling between the two internal states. Such systems have attracted considerable attention in mathematical physics and PDE theory, since they exhibit rich phenomena such as phase separation, symmetry breaking, and interface dynamics \cite{AmbrosettiColorado2007,BaoCai2013, WeiWeth2010}.

In the one-dimensional case ($n=1$), the system has been analyzed in \cite{AftalionSourdis2019}, where the existence and asymptotic properties of domain wall solutions were established under different regimes of the parameter $a$. This analysis revealed the heteroclinic structure of the problem and its connection to the minimization of the Gross-Pitaevskii energy. In \cite{Fazly2021JDE,Fazly2013CVPDE}, the authors proved the one-dimensional symmetry of solutions to the system without a Rabi frequency in dimensions up to $3$. Such results provide a natural motivation to study the monotonicity and uniqueness of solutions in higher dimensions. From a physical perspective, understanding the qualitative properties of \eqref{eq. main} is also a first step towards the mathematical analysis of the vortex patterns appearing in Rabi-coupled condensates, in particular the so-called multidimer bound states \cite{AftalionMason2016,CiprianiNitta2013, KobayashiNitta2014}: these are molecular-type states formed by the binding of vortices of different components, which then interact to produce a variety of complex patterns, such as honeycomb, triangular, or square lattices. The possibility of non-segregated states at infinity with positive limits lies at the origin of such rich structures.

From the mathematical viewpoint, the system \eqref{eq. main} provides a natural framework for investigating qualitative properties of solutions. In particular, when $\alpha=2$ and $0\leq\omega<1$, it decouples into a single semilinear equation closely related to the Allen-Cahn equation, whose monotonicity and symmetry properties are central to the celebrated De Giorgi conjecture. Therefore, the study of \eqref{eq. main} can be regarded as a natural extension of the De Giorgi problem from single equations to coupled systems, thus bridging mathematical physics and geometric analysis.

\subsection{Classical theory in a special case \texorpdfstring{$\alpha=2$}{Lg}}
In the case $\alpha=2$, the system \eqref{eq. main} can be decoupled and reduced to the classical Allen-Cahn equation \eqref{eq:AC} (see  Corollary~\ref{cor. decouple corollary} for details). This could be viewed as a critical point of the classical Ginzburg-Landau energy:
\begin{equation}\label{eq. classical GL energy}
    J(u,\Omega)=\int_{\Omega}\Big\{\frac{|\nabla u|^{2}}{2}+W(u)\Big\}dx,\quad W(u)=\frac{1}{4}(1-u^{2})^{2}\chi_{[-1,1]}.
\end{equation}
The functional \eqref{eq. classical GL energy} was originally developed from the theory of Van der Waals (see \cite{Row79}) by Landau, Ginzburg, and Pitaevskii \cite{GP58,Landau37,Landau67} in order to describe phase transitions in thermodynamics.

It is immediate to observe that \eqref{eq. classical GL energy} admits two stable constant states, namely $u\equiv\pm 1$. A more interesting problem, however, is the study of phase transition solutions, which interpolate between values close to $-1$ and $1$ within a relatively narrow transition layer, referred to as the phase field region.

It is now well established that energy-minimizing phase transition solutions are closely connected to minimal surfaces. Indeed, consider the rescaled solution
 $u_{R}(y):=u(Ry)$, which is a minimizer of the rescaled Ginzburg-Landau energy functional:
\begin{equation*}
    J_{R}(u_{R},\Omega)=\int_{\Omega}\Big\{\frac{|\nabla u_{R}|^{2}}{2R}+RW(u_{R})\Big\}dy.
\end{equation*}
It was shown by Modica and Mortola in \cite{M79,MM80} (see also \cite{Bou90,OS91,SV12}) that as $R\to\infty$, the rescaled Ginzburg-Landau energy functional $\Gamma$-converges to the perimeter functional of Caccioppoli sets. As a result, a subsequence of the rescaled solutions $u_{R}$'s converges to a set characteristic function in the form
\begin{equation*}
    u_{\infty}=\chi_{E}-\chi_{E^{c}}
\end{equation*}
in the weak BV sense and in the strong $L^{1}_{loc}$ sense. Here, $E$ is a Caccioppoli set with minimal perimeter.

Equally important is the fact that
 the boundary of the limiting Caccioppoli set $E$ contains the origin. 
 This is a consequence of
 a density estimate result obtained by Caffarelli and C\'ordoba in \cite{CC95} (see also \cite{DFGV25a,DFGV25b,DFV18,FV08,PV05b,PV05a,SV14,SZ25,V04,Z25}), which states that the volume of the positive set and the negative set of $u(x)$ must be comparable. Combining the $\Gamma$-convergence and the density estimate  with the Bernstein theorem for minimal surfaces obtained in Simons \cite{S68}, one deduces that the zero set of a global energy minimizer $u(x)$ in $\mathbb{R}^{n}$ must be asymptotically flat, provided  that $n\leq7$.

For such a single equation, there have been a number of important results related to the classification of global solutions. In \cite{GG98}, Ghoussoub and Gui proved the De Giorgi conjecture in $\mathbb{R}^{2}$, and then Ambrosio and Cabr\'e \cite{AC00} showed the same result in $\mathbb{R}^{3}$.
In dimensions $\mathbb{R}^{4}-\mathbb{R}^{8}$, the validity of the De Giorgi conjecture remains open, except under additional assumptions.
 For example, Ghoussoub and Gui \cite{GG03} proved the flatness of solutions under an additional odd-symmetry assumption. Besides, in \cite{S09}, Savin proved the De Giorgi conjecture in $\mathbb{R}^{8}$, with the additional assumption that $u(x',x_{n})\to\pm1$ non-uniformly as $x_{n}\to\pm\infty$. In fact, the asymptotic flatness of the level set implies the flatness of a global solution. Later, the convergence assumption in \cite{S09} was replaced by the assumption that some level set of $u$ is an entire graph by Farina and Valdinoci in \cite{Farina2011TAMS}.

These classification results have also  been extended to more general settings. In \cite{SV05,VSS06}, as well as the previously mentioned \cite{Farina2011TAMS}, the authors established the one-dimensional symmetry for solutions to the $p$-Laplacian Allen-Cahn equation.
In \cite{CC12,DSV20,S18a,S18b}, the De Giorgi conjecture was verified for the fractional Allen-Cahn equation in various settings.

In higher dimensions, however, the situation is completely different.
In particular, del Pino, Kowalczyk, and Wei \cite{dPKW11} constructed a nontrivial global solution in $\mathbb{R}^{9}$.
Their construction was based on Simons' cone \cite{BdGG69,S68}, the most well-known non-flat stable minimal graph, which arises only in high dimensions.
In other words, problems of De Giorgi conjecture type typically face a dimensional obstruction. Thanks to the example in \cite{dPKW11}, one can easily construct a non-trivial solution to \eqref{eq. main} in $\mathbb{R}^{9}$, see Corollary~\ref{cor. decouple corollary}.

\subsection{Gibbons' conjecture: similarities and differences}
The so-called \emph{Gibbons conjecture} \cite{Carbou1995} asserts that bounded entire solutions $u$ of the Allen-Cahn equation~\eqref{eq:AC} in $\mathbb{R}^{n}$, which converge to $\pm 1$ as $x_{n} \to \pm \infty$, uniformly with respect to the other variables, must necessarily be one-dimensional. In other words, the asymptotic flatness of the level sets at infinity enforces global one-dimensional symmetry.

The introduction of the Gibbons conjecture is natural since  it provides an alternative formulation of the De Giorgi conjecture under stronger asymptotic boundary conditions. While the De Giorgi conjecture requires monotonicity in a given direction, the Gibbons conjecture assumes uniform limits at infinity, which are often easier to verify in applications to phase transitions and mathematical physics. In this sense, the two conjectures complement each other: the De Giorgi conjecture emphasizes internal monotonicity, whereas the Gibbons conjecture highlights the asymptotic behavior at infinity. Both have played a fundamental role in the development of the theory of symmetry and rigidity for nonlinear elliptic equations.

Unlike the De Giorgi conjecture, for which the critical dimension is $8$, the validity of the Gibbons conjecture has been established in all dimensions by independent methods due to Barlow, Bass and Gui \cite{BarlowBassGui2000}, Berestycki and Hamel \cite{BF2007}, Berestycki, Hamel and Monneau \cite{BerestyckiHamelMonneau2000}, and Farina \cite{Farina1999}. Farina's contribution \cite{Farina1999} was the first in this direction, and his work did not require the boundedness of the solution. For the fractional case, we refer interested readers to \cite{WC2022,WC2024}.

Analogous results for non-cooperative elliptic systems can be found in \cite{FarinaSciunziSoave2020,Le2023}.
As for the competitive system \eqref{eq. main} related to two-component Bose-Einstein condensates, its Gibbons' conjecture was proven recently in \cite{AFN21} by Aftalion, Farina, and Nguyen in all dimensions.

The uniform convergence assumption is crucial in Gibbons' conjecture. Violation of such an assumption, even the violation of uniformity (like the assumption in \cite{S09}), might result in the existence of non-trivial solutions. For example, the non-flat solution example in \cite{dPKW11} violates the uniform convergence assumptions in Gibbons' conjecture. We will also provide a non-trivial solution to the system \eqref{eq. main} in Corollary~\ref{cor. decouple corollary} below, whose convergence is also non-uniform.

\subsection{Other comments}
We plan to study the De Giorgi conjecture for a nonlocal version of \eqref{eq. main}, involving the operator $(-\Delta)^{s}$, in the second paper of this series. We conjecture that the De Giorgi property for the two-component Bose-Einstein system holds in $\mathbb{R}^{3}$ when $s \ge \frac{1}{2}$, and in $\mathbb{R}^{2}$ when $s < \frac{1}{2}$.

To the best of our knowledge, no density estimates or $\Gamma$-convergence results have yet been established for systems analogous to \eqref{eq. main}. Nevertheless, we anticipate that versions of the classical density estimate and $\Gamma$-convergence results could be developed for \eqref{eq. main}, thereby implying the asymptotic flatness of the phase field region in $\mathbb{R}^{8}$.

\section{Main results and key ideas}
\begin{figure}[htbp]
    \centering
    \begin{tikzpicture}[scale=5,>=latex]
      \draw[->,thick] (0,0) -- (1.8,0) node[font=\Large][right] {$x$};
      \draw[->,thick] (0,0) -- (0,1.6) node[font=\Large][right] {$y$};
       \draw[thick] (1,0) -- (1,0.01) node[below] {$1$};
\draw[thick] (0, 1) -- (0.01,1) node[left] {$1$};
      \node[below left] at (0,0) {$0$};

      \pgfmathsetmacro{\aone}{sqrt(0.5 - sqrt(3)/4)}
      \pgfmathsetmacro{\bone}{1/(4*\aone)}
      \pgfmathsetmacro{\atwo}{sqrt(0.5 + sqrt(3)/4)}
      \pgfmathsetmacro{\btwo}{1/(4*\atwo)}

\draw[white, dashed] plot[domain=0:\atwo, variable=\x, smooth] ({\x},{sqrt(1-\x*\x)});
\draw[white, dashed] plot[domain=\aone:\atwo, variable=\x, smooth] ({\x},{1/(4*\x)});
      \begin{scope}
        \clip plot[domain=0:\atwo, variable=\x, smooth] ({\x},{sqrt(1-\x*\x)})
             -- plot[domain=\atwo:\aone, variable=\x, smooth] ({\x},{1/(4*\x)}) -- cycle;
        \fill[pattern=north east lines, pattern color=gray!40] (0,0) rectangle (1.6,1.6);
      \end{scope}

      \draw[dashed, thick, domain=0:1, variable=\x, smooth]
        plot ({\x},{sqrt(1-\x*\x)}) node[font=\large][left] at (-0.05, 1.1) {$x^2+y^2=1$};

      \draw[dashed, red, thick, domain=0.17:1.5, variable=\x, smooth]
        plot ({\x},{1/(4*\x)}) node[font=\large][right] {$xy=\tfrac{\omega}{\alpha}$};

      \filldraw[blue] (\aone,\bone) circle(0.6pt) node[font=\large][above right] {$(a,b)$};
      \filldraw[blue] (\atwo,\btwo) circle(0.6pt) node[font=\large][above right] {$(b,a)$};

      \filldraw[green!80!black] (0.67,0.67) circle(0.6pt);
      \node[green, font=\fontsize{12}{12}\selectfont] at (1,0.75) 
      {$\Bigl(\sqrt{\tfrac{1+\omega}{2+\alpha}},\,\sqrt{\tfrac{1+\omega}{2+\alpha}}\Bigr)$};

    \end{tikzpicture}
\caption{Steady states $(a, b), (b, a)$ and  $\Bigl(\sqrt{\tfrac{1+\omega}{2+\alpha}},\,\sqrt{\tfrac{1+\omega}{2+\alpha}}\Bigr)$.}\label{fig. range}
\end{figure}
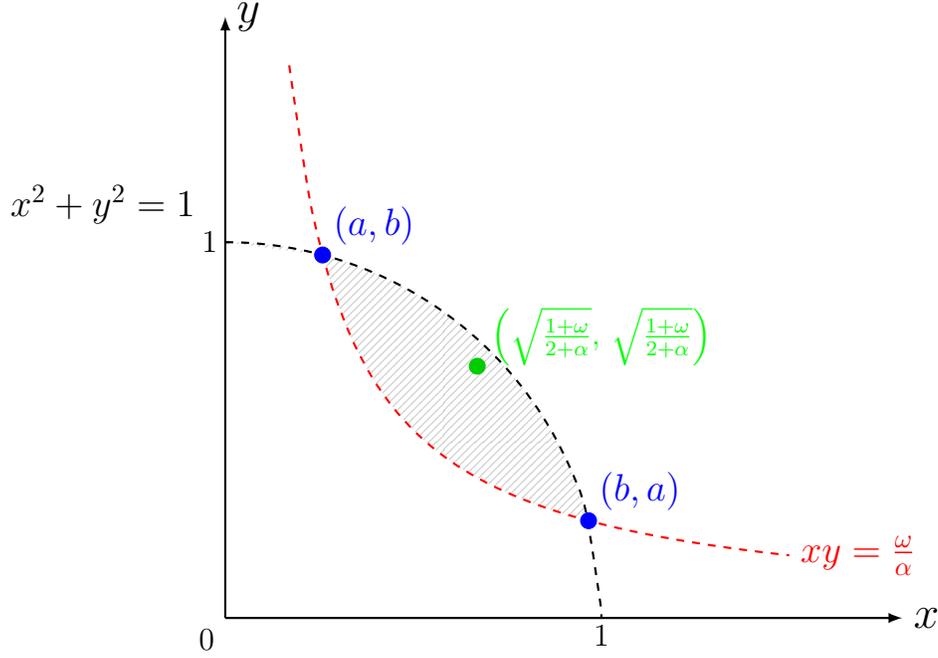

In this section, we present the main results of this paper. We begin by introducing some notational conventions. Let $(u,v)$ be an entire positive solution to \eqref{eq. main} in $\mathbb{R}^{n}$, with $0 < \omega < \frac{\alpha}{2}$. As shown in \cite[Proposition 2.1]{AFN21}, the range of $(u,v)$ lies within the region bounded by the circle $x^{2}+y^{2}=1$ and the hyperbola $xy = \frac{\omega}{\alpha}$, corresponding to the shaded area in Figure~\ref{fig. range}.

One can verify that the system \eqref{eq. main} admits three distinct positive steady states, i.e., constant solutions $(u,v)$ of \eqref{eq. main}. Specifically, let $0 < a < b$ be such that $(a,b)$ and $(b,a)$ are the intersections of $x^{2}+y^{2}=1$ and $xy = \frac{\omega}{\alpha}$ in the first quadrant (see Figure~\ref{fig. range}). Then $(a,b)$ and $(b,a)$ correspond to the two stable steady states of the system. The third steady state is $(\sqrt{\tfrac{1+\omega}{2+\alpha}}, \sqrt{\tfrac{1+\omega}{2+\alpha}})$, which should be regarded as unstable in view of the energy functional \eqref{eq. Ginzburg-Landau energy, s=1}, to be discussed later.

We first focus on the special case $\alpha = 2$ and present a classification result. This result is more of an observation than a full theorem, as it relies on the simple fact that the system \eqref{eq. main} decouples when $\alpha = 2$. By applying well-known classification results for a single equation, we obtain the following corollary.

\begin{corollary}\label{cor. decouple corollary}
    Let $\alpha=2$ and $0<\omega<1$. Then all global positive solutions to \eqref{eq. main}  satisfy $u(x)+v(x)\equiv\sqrt{1+\omega}$. Moreover, let $0\leq a<b$ be  constants satisfying $a^{2}+b^{2}=1$ and $\displaystyle ab=\frac{\omega}{2}$. Assume that $\frac{\partial u}{\partial x_{n}}>0>\frac{\partial v}{\partial x_{n}}$, and  that the following convergence holds (not uniform in $x'$):
    \begin{equation*}
        \bigl(u(x',x_{n}),v(x',x_{n})\bigr)\to\left\{\begin{aligned}
            &(a,b),&\mbox{as }&x_{n}\to-\infty,\\
            &(b,a),&\mbox{as }&x_{n}\to+\infty,
        \end{aligned}\right.
    \end{equation*}
    then we have the following results:
    \begin{itemize}
        \item If $n\leq8$, then the pair $(u,v)$ is one-dimensional.
        \item If $n\geq9$, then there exists a pair $(u,v)$ that  is not one-dimensional.
    \end{itemize}
\end{corollary}
\begin{remark}
    In Farina-Sciunzi-Soave \cite{FarinaSciunziSoave2020}, the authors have already realized that the system \eqref{eq. main} can be decoupled in the case $\omega=0$, and they have obtained a similar result as in Corollary~\ref{cor. decouple corollary}.
\end{remark}

Let us state our main theorem, which is a Liouville-type result of De Giorgi type in $\mathbb{R}^{3}$.

\begin{theorem}\label{thm. main theorem}
Let $(\alpha,\omega)$ be a pair of parameters satisfying $0 < \omega < \frac{\alpha}{2}$. Assume that $(u,v)$ is a pair of positive functions satisfying \eqref{eq. main} in $\mathbb{R}^{3}$. If
\begin{equation*}
    \frac{\partial u}{\partial x_3} > 0 > \frac{\partial v}{\partial x_3} \quad \text{in } \mathbb{R}^{3},
\end{equation*}
then the pair $(u,v)$ is one-dimensional. Precisely, there exists a unit vector $\vec{\nu} \in \mathbb{R}^{3}$ and a pair of positive functions $(U(t),V(t))$ satisfying \eqref{eq. main} in $\mathbb{R}$ such that
\begin{equation*}
    (u(x),v(x)) \equiv (U(x \cdot \vec{\nu}), V(x \cdot \vec{\nu})), \quad \text{with } U'(t) > 0 > V'(t) \text{ in } \mathbb{R}.
\end{equation*}
Moreover, let $0 < a < b$ be constants such that $a^2 + b^2 = 1$ and $ab = \frac{\omega}{\alpha}$ (see Figure~\ref{fig. range}). Then
\begin{equation*}
    \lim_{t \to -\infty} (U(t), V(t)) = (a,b), \quad
    \lim_{t \to +\infty} (U(t), V(t)) = (b,a).
\end{equation*}
\end{theorem}

Our method is inspired by Ambrosio and Cabr\'e \cite{AC00}, and we focus on proving that the slopes of the level sets of $(u,v)$ are constants. This method is also applied to the Gross-Pitaevskii system without the Rabi frequency by Fazly and Ghoussoub in \cite{Fazly2021JDE,Fazly2013CVPDE}. Specifically, let $(\xi_i, \eta_i)$ denote the directional derivatives of $(u,v)$ in the $\vec{e}_i$-direction; these satisfy the linearized equation of \eqref{eq. main}. Since we assume that 
\[
\frac{\partial u}{\partial x_3} > 0 > \frac{\partial v}{\partial x_3},
\]
it follows that $(\varphi, \psi) := (\xi_3, \eta_3)$ is a nowhere-zero pair of functions. By the implicit function theorem, the quantities
\[
(\sigma_i, \tau_i) := \left( \frac{\xi_i}{\varphi}, \frac{\eta_i}{\psi} \right)
\]
represent the slopes of the level sets of $(u,v)$. By deriving a divergence-form equation for $(\sigma_i, \tau_i)$ (see Lemma~\ref{lem. negative definite form}) and applying the logarithmic cut-off trick, we show that $(\sigma_i, \tau_i)$ are constants satisfying $\sigma_i = \tau_i$, thereby establishing that $(u,v)$ is one-dimensional.

Compared with \cite{AC00}, where the authors considered a single equation, the main difficulty here arises from the coupling between the two components of \eqref{eq. main}. Fortunately, using the range estimate from \cite{AFN21}, our divergence-form equation for $(\sigma_i, \tau_i)$ in Lemma~\ref{lem. negative definite form} takes a negative definite form:
\[
-\mathcal{P}(x) (\sigma_i - \tau_i)^2,
\]
where $\mathcal{P}(x)$ is a strictly positive function. Consequently, not only are $\sigma_i$ and $\tau_i$ constant, but these constants must also be equal.

We also note that \eqref{eq. main} possesses a variational structure. To be precise, let us define an infinitesimal potential energy by
\begin{equation}\label{eq. double well infinitesimal potential energy}
    W(u,v)=\frac{1}{4}(1-u^{2}-v^{2})^{2}+\frac{1}{2\alpha}(\omega-\alpha uv)^{2}.
\end{equation}
It follows that for such an infinitesimal potential energy \eqref{eq. double well infinitesimal potential energy}, we have
\begin{equation*}
    u(u^{2}+v^{2}-1)+v(\alpha uv-\omega)=\frac{\partial W}{\partial u}(u,v),\quad v(u^{2}+v^{2}-1)+u(\alpha uv-\omega)=\frac{\partial W}{\partial v}(u,v).
\end{equation*}
In other words, \eqref{eq. main} can be rewritten as
\begin{equation}\label{eq. simplified B-E}
    \Delta u=W_{u}(u,v),\quad\Delta v=W_{v}(u,v).
\end{equation}

It is therefore natural to define a coupled Ginzburg-Landau energy functional similar to \eqref{eq. classical GL energy}. Let $\Omega$ be an arbitrary bounded open domain in $\mathbb{R}^{n}$, and define
\begin{equation}\label{eq. Ginzburg-Landau energy, s=1}
    J(u,v,\Omega)=\int_{\Omega}\Big\{\frac{|\nabla u|^{2}+|\nabla v|^{2}}{2}+W(u,v)\Big\}dx.
\end{equation}
It is readily seen that \eqref{eq. simplified B-E} is precisely the Euler-Lagrange equation of the Ginzburg-Landau energy \eqref{eq. Ginzburg-Landau energy, s=1}. In this paper,  the estimate of this Ginzburg-Landau energy \eqref{eq. Ginzburg-Landau energy, s=1} will play a crucial role in establishing the one-dimensional symmetry of $(u,v)$.

The estimate of the Ginzburg-Landau energy proceed in two main steps. 
In the first step, we consider a cube of radius $R$ and translate it to infinity along the $\vec{e}_3$ direction, estimating the change in energy during this process.
 The second step is to analyze the limiting profile as the solution $(u,v)$ is  translated along the $\vec{e}_{3}$ direction to infinity. 
  Since both $u$ and $v$ are assumed monotone in the $\vec{e}_3$ direction, the $\vec{e}_3$-translation of $(u,v)$ converges to a limiting profile $(\overline{u},\overline{v})$ or $(\underline{u},\underline{v})$ at $\pm\infty$, which are stable solutions of \eqref{eq. main} in the lower-dimensional space $\mathbb{R}^2$. 
We emphasize that, when constructing a suitable nowhere-zero pair $(\varphi,\psi)$, we again utilize the range estimate from \cite[Proposition 2.1]{AFN21} and apply the strong maximum principle for linear elliptic systems, thereby ensuring that $(\varphi,\psi)$ is nowhere zero.

We conclude this section with a brief proof of Corollary~\ref{cor. decouple corollary}.

\begin{proof}[Proof of Corollary~\ref{cor. decouple corollary}]
    Let $\mathcal{S}=(u+v-\sqrt{1+\omega})$. By adding the two lines of \eqref{eq. main}, we obtain
    \begin{equation}\label{eq. S equation, to decouple}
        \Delta\mathcal{S}=(\mathcal{S}+\sqrt{1+\omega})(\mathcal{S}+2\sqrt{1+\omega})S=:c(\mathcal{S})\mathcal{S}.
    \end{equation}
   From  \cite[Proposition 1.2]{AFN21}, we know that
    $$u^{2}+v^{2}\leq1 \quad \mbox{and} \quad \displaystyle uv\geq\frac{\omega}{\alpha}=\frac{\omega}{2},$$
   which implies  that $c(\mathcal{S})\geq2\omega+\sqrt{2\omega+2\omega^{2}}>0$. We claim that $\mathcal{S}\equiv0$. To see this, suppose by contradiction  that $\mathcal{S} \not\equiv 0$, and consider
    \begin{equation*}
        \mathcal{E}(R):=\int_{B_{R}}\Big\{|\nabla\mathcal{S}|^{2}+2\omega\mathcal{S}^{2}\Big\}dx=\int_{B_{R}}\Big\{\Delta\frac{\mathcal{S}^{2}}{2}-\mathcal{S}\Delta\mathcal{S}+2\omega\mathcal{S}^{2}\Big\}dx,
    \end{equation*}
    which is strictly positive for a sufficiently large $R$. Since $c(\mathcal{S})>2\omega$ in \eqref{eq. S equation, to decouple}, we have
    \begin{equation*}
        \mathcal{E}(R)\leq\int_{B_{R}}\Delta\frac{\mathcal{S}^{2}}{2}dx=\int_{\partial B_{R}}\mathcal{S}\cdot\partial_{r}\mathcal{S}d\mathcal{H}^{n-1}_{\partial B_{R}}\leq\int_{\partial B_{R}}\frac{\mathcal{S}^{2}+|\nabla\mathcal{S}|^{2}}{2}d\mathcal{H}^{n-1}_{\partial B_{R}}.
    \end{equation*}
    In other words, set $\lambda:=\min\{2,4\omega\}$, then we have
    \begin{equation*}
        \frac{d}{dr}\mathcal{E}(R)\geq\lambda\mathcal{E}(R).
    \end{equation*}
    By the Caccioppoli inequality, we see that $\|\mathcal{S}\|_{L^{\infty}(B_{R})}$ must grow at least exponentially in $R$, thus violating the facts $u^{2}+v^{2}\leq1$ and $\displaystyle uv\geq\frac{\omega}{2}$. This means that $\mathcal{S}\equiv0$.

    Then, define the difference:
    \begin{equation*}
        \mathcal{D}(x)=\frac{u(x)-v(x)}{b-a}.
    \end{equation*}
    Subtracting the two lines of \eqref{eq. main}, we obtain 
    \begin{equation*}
        \Delta\mathcal{D}=(1-\omega)(\mathcal{D}^{3}-\mathcal{D}),
    \end{equation*}
       where we  used the fact that $(b-a)^2=(a^{2}+b^{2})-2ab=1-\omega.$
    After an appropriate scaling, this is precisely  the classical Allen-Cahn equation \eqref{eq:AC}. 
     The one-dimensional symmetry in dimensions $n \leq 8$ follows from \cite{S09}, while the non-flat example in dimensions $n \geq 9$ follows from the construction in \cite{dPKW11}.
\end{proof}

One of the key innovations of this paper is the combination of energy analysis with De Giorgi-type classification results, providing a deep understanding of the structure and stability of solutions in the coupled two-component Bose-Einstein condensate system. In particular, we introduce a novel method of proving the one-dimensional symmetry of solutions by constructing appropriate directional derivative symmetries and utilizing range estimates. This approach is crucial in analyzing the asymptotic behavior of the solutions and contributes to a deeper understanding of the system's solution structure, especially in higher-dimensional settings. 

Another key contribution lies in the introduction of an energy functional analysis, which serves as a powerful tool for handling the system's stability and structural behavior. This approach not only provides a new perspective on the classical Ginzburg-Landau functional but also reveals the existence of nontrivial solutions in higher dimensions. Additionally, the classification of boundary conditions and steady-state solutions, which was previously underexplored in the literature, has been thoroughly analyzed in this study. Through detailed analysis of steady-state solution structures and convergence behavior, we unveil multiple solution structures under different parameter regimes, extending the classical Liouville-type theory to coupled systems.

From a methodological perspective, although De Giorgi-type classification results have been well-developed for single equations, handling coupled systems presents significant challenges due to the nonlinearity and interaction between components. Our breakthrough lies in combining symmetry-breaking and asymptotic analysis methods, which allow us to analyze the behavior of the solutions and prove their one-dimensional symmetry in $\mathbb{R}^3$. This method is particularly valuable for systems with complex interactions, such as Bose-Einstein condensates, and extends the applicability of De Giorgi-type classification results to higher-dimensional multi-component systems.

This work not only provides a classification of solutions for coupled systems but also opens new avenues for analyzing the behavior of solutions in high-dimensional spaces, which has not been fully explored in the literature. Furthermore, our findings offer a solid theoretical foundation for future studies of multi-component systems with nonlocal effects or interactions, particularly in the context of mathematical physics and the theory of phase transitions.

\section{The linearized equation and the Liouville theorem}
From now on, we aim at proving Theorem~\ref{thm. main theorem}. In this section, 
we consider the linearized equation of \eqref{eq. main} (or equivalently \eqref{eq. simplified B-E}), which takes the form of a Schr\"odinger-type equation.

\subsection{The Schr\"odinger equation}
Let $\vec{e}\in\partial B_{1}$ be a unit vector, and  differentiate \eqref{eq. simplified B-E} in the $\vec{e}$ direction. The following lemma gives the linearized equation of \eqref{eq. simplified B-E}.
\begin{lemma}\label{lem. derivative satisfies the Schrodinger equation}
    Let $(u,v)$ be a solution to \eqref{eq. simplified B-E} and let $\vec{e}\in\partial B_{1}$ be a unit vector, then $(\partial_{\vec{e}}u,\partial_{\vec{e}}v)$ satisfies the following Schr\"odinger-type  equation:
    \begin{equation}\label{eq. Schrodinger of derivative}
        \left\{\begin{aligned}
            -\Delta\partial_{\vec{e}}u+W_{uu}(u,v)\partial_{\vec{e}}u+W_{uv}(u,v)\partial_{\vec{e}}v=0,\\
            -\Delta\partial_{\vec{e}}v+W_{vu}(u,v)\partial_{\vec{e}}u+W_{vv}(u,v)\partial_{\vec{e}}v=0.
        \end{aligned}\right.
    \end{equation}
\end{lemma}
\begin{proof}
    Differentiating \eqref{eq. simplified B-E} in the $\vec{e}$ direction, we obtain
    \begin{equation*}
        \Delta\partial_{\vec{e}}u=\partial_{\vec{e}}\Delta u=\partial_{\vec{e}}W_{u}(u,v).
    \end{equation*}
    Applying the chain rule then yields the first equation in \eqref{eq. Schrodinger of derivative}. The second equation can be derived in a similar manner.
\end{proof}

Motivated by Lemma~\ref{lem. derivative satisfies the Schrodinger equation}, it is natural to define the following Schr\"odinger operator:
\begin{definition}
Let $(u,v)$ be a solution of \eqref{eq. simplified B-E}. We define the Schr\"odinger operator associated with $(u,v)$ as 
\[
L : C^{\infty}(\mathbb{R}^n) \times C^{\infty}(\mathbb{R}^n) \to C^{\infty}(\mathbb{R}^n) \times C^{\infty}(\mathbb{R}^n)
\]
by
   \begin{equation}\label{eq. Schrodinger}
        L\begin{bmatrix}\xi\\\eta\end{bmatrix}:=\Big\{-\Delta+D^{2}W\Big\}\begin{bmatrix}\xi\\\eta\end{bmatrix}=\begin{bmatrix}-\Delta\xi+W_{uu}(u,v)\xi+W_{uv}(u,v)\eta\\-\Delta\eta+W_{vu}(u,v)\xi+W_{vv}(u,v)\eta\end{bmatrix}.
    \end{equation}
\end{definition}

Next, we define a quadratic operator associated with the ratio.

\begin{definition}
Let $(\varphi, \psi) \in C^{\infty}(\mathbb{R}^n) \times C^{\infty}(\mathbb{R}^n)$ satisfy $\varphi > 0 > \psi$ everywhere in $\mathbb{R}^n$. Let $(\sigma, \tau)$ be a pair of smooth functions such that $(\xi, \eta) := (\varphi \sigma, \psi \tau)$ belongs to $C^{\infty}(\mathbb{R}^n) \times C^{\infty}(\mathbb{R}^n)$ with all derivatives bounded. Then we define
\begin{align}\label{eq. ratio quadratic operator}
Q_{(\varphi, \psi)} \begin{bmatrix} \sigma \\ \tau \end{bmatrix} 
:=& -\xi \Delta \xi + \sigma^2 \varphi \Delta \varphi - \eta \Delta \eta + \tau^2 \psi \Delta \psi \\
=& -\sigma \, \mathrm{div}(\varphi^2 \nabla \sigma) - \tau \, \mathrm{div}(\psi^2 \nabla \tau).
\end{align}
\end{definition}

In fact, the quadratic form \eqref{eq. ratio quadratic operator} always acts as a negative semi-definite form in the context of this paper.

\begin{lemma}\label{lem. negative definite form}
Assume that $\varphi > 0 > \psi$ everywhere, and that 
\[
L \begin{bmatrix} \varphi \\ \psi \end{bmatrix} \in \overline{\mathbb{R}_+} \times \overline{\mathbb{R}_-} \quad \text{(the fourth quadrant).}
\]
Let $(\xi, \eta)$ belong to the kernel of \eqref{eq. Schrodinger}, and define 
\[
(\sigma, \tau) := \Big(\frac{\xi}{\varphi}, \frac{\eta}{\psi}\Big).
\]
Then there exists a strictly positive function $\mathcal{P}(x)$ such that
\begin{equation}\label{eq. Q is negative semi-definite}
Q_{(\varphi, \psi)} \begin{bmatrix} \sigma \\ \tau \end{bmatrix} \leq -(\sigma - \tau)^2 \mathcal{P}(x) \quad \text{everywhere in } \mathbb{R}^n.
\end{equation}
\end{lemma}

\begin{proof}
By the positivity assumptions on $(\varphi, \psi)$, we have
\begin{equation*}
\varphi \Delta \varphi \leq W_{uu} \varphi^2 + W_{uv} \varphi \psi, \quad 
\psi \Delta \psi \leq W_{vu} \varphi \psi + W_{vv} \psi^2.
\end{equation*}
It follows that
\begin{align*}
Q_{(\varphi, \psi)} \begin{bmatrix} \sigma \\ \tau \end{bmatrix} 
&\leq -\big( W_{uu} \xi^2 + W_{uv} \xi \eta \big) + \sigma^2 (W_{uu} \varphi^2 + W_{uv} \varphi \psi) \\
&\quad - \big( W_{vu} \xi \eta + W_{vv} \eta^2 \big) + \tau^2 (W_{vu} \varphi \psi + W_{vv} \psi^2) \\
&= W_{uu} (\sigma^2 \varphi^2 - \xi^2) + W_{vv} (\tau^2 \psi^2 - \eta^2) \\
&\quad + W_{uv} \big( -\xi \eta + \sigma^2 \varphi \psi - \xi \eta + \tau^2 \varphi \psi \big) \\
&= W_{uv} \varphi \psi (\sigma - \tau)^2 = - (\sigma - \tau)^2 \mathcal{P}(x), 
\end{align*}
with $\mathcal{P}(x) := - W_{uv} \varphi \psi.$

To show that $\mathcal{P}(x)$ is strictly positive, it suffices to verify that $W_{uv} > 0$. Direct computation yields
\begin{equation}\label{eq. W_uv>0}
W_{uv}(u,v) = (2 + 2\alpha) uv - \omega \geq \Big( \frac{2}{\alpha} + 1 \Big) \omega > 0,
\end{equation}
where we have used the range estimate from \cite[Proposition 1.2]{AFN21}: $\displaystyle uv \geq \frac{\omega}{\alpha}$.
\end{proof}

\subsection{The Caccioppoli inequality and the Liouville theorem}

In this subsection, we prove the following Liouville-type theorem using the Caccioppoli inequality. A logarithmic cut-off trick will be employed when selecting the cut-off function.

\begin{lemma}[Liouville-type result]\label{lem. Liouville-type result with cut-off function}
Let $\varphi, \psi, \xi, \eta, \sigma, \tau \in C^\infty(\mathbb{R}^n)$ satisfy
\begin{align}\label{eq1}
Q_{(\varphi, \psi)} \begin{bmatrix} \sigma \\ \tau \end{bmatrix} 
&= -\xi \Delta \xi + \sigma^2 \varphi \Delta \varphi - \eta \Delta \eta + \tau^2 \psi \Delta \psi \\
&= -\sigma \cdot \mathrm{div}(\varphi^2 \nabla \sigma) - \tau \cdot \mathrm{div}(\psi^2 \nabla \tau) \\
&\leq - (\sigma - \tau)^2 \mathcal{P}(x),
\end{align}
where $\varphi > 0 > \psi$ and $(\xi, \eta) = (\varphi \sigma, \psi \tau)$. 

If, for all $R > 2$, the following growth conditions hold:
\begin{equation}\label{eq. quadratic energy growth condition}
\int_{B_R} \xi^2 \, dx \le C_0 R^2, 
\quad 
\int_{B_R} \eta^2 \, dx \le C_0 R^2,
\end{equation}
then $\sigma(x)$ and $\tau(x)$ are constant in $\mathbb{R}^n$, with $\sigma = \tau$.
\end{lemma}

\begin{proof}
For $R>2$, define a Lipschitz cut-off function 
\[
\zeta_R(x) =
\begin{cases}
1,& |x|\le R,\\[1mm]
\dfrac{\ln(R^2/|x|)}{\ln R}, & R<|x|<R^2,\\[1mm]
0,& |x|\ge R^2.
\end{cases}
\]
Then $\zeta_R$  satisfies
\begin{equation}\label{eq:grad-bound}
|\nabla \zeta_R(x)| = \frac{1}{|x|\ln R}, \quad \mbox{for}\,\,R<|x|<R^2.
\end{equation}

Denote the annular region $A_R=\{x\in \mathbb{R}^3 \mid R<|x|<R^2\}$. 
We further decompose $A_R$ dyadically:
\[
A_R \subset \bigcup_{k=0}^{K+1} A_R^k,\qquad A_R^k:=\{x\in \mathbb{R}^3 \mid 2^k R<|x|\le 2^{k+1}R\},\quad K=\lfloor\log_2 R\rfloor,
\]
where $\lfloor x \rfloor$ denotes the floor function.

Using the energy growth condition \eqref{eq. quadratic energy growth condition} and  \eqref{eq:grad-bound}, we obtain 
\begin{equation}\label{eq:g1}
\begin{aligned}
    \int_{A_R} \xi^2 |\nabla \zeta_R|^2 dx\leq &  \sum_{k=0}^{K+1} \int_{A_R^k} \xi^2 |\nabla \zeta_R|^2 dx\leq C_0 \sum_{k=0}^{K+1}\Big( \frac{1}{(2^kR\ln R)^2}  (2^{k+1}R)^2  \Big)\\
    =&\frac{4C_0}{(\ln R)^2}(K+2)\leq\frac{C}{\ln R}.
\end{aligned}
\end{equation}
Similarly, we have
\begin{equation}\label{eq:g2}
    \int_{A_R} \eta^2 |\nabla \zeta_R|^2 dx  \leq  \frac{C}{\ln R}.
\end{equation}

Multiplying \eqref{eq1} by $\zeta_R^2(x)$ and integrating over $\mathbb R^3$ yields
\begin{equation*}
\begin{aligned}
0 \ge & \int_{\mathbb R^3}\big(-\sigma\cdot\mathrm{div}(\varphi^{2}\nabla\sigma)-\tau\cdot\mathrm{div}(\psi^{2}\nabla\tau)\big)\zeta_R^2(x)\,dx\\
=&\int_{{\mathbb R}^3}\big(\varphi^{2} \nabla\sigma \cdot\nabla (\sigma \zeta_R^2)+\psi^{2}\nabla\tau \cdot\nabla (\tau \zeta_R^2)\big)\,dx\\
= & \int_{B_{R^2}}\big(\varphi^{2}|\nabla(\zeta_{R}\sigma)|^2+\psi^{2}|\nabla(\zeta_{R}\tau)|^2\big)\,dx
-\int_{B_{R^2}}\big(\varphi^2\sigma^{2}|\nabla\zeta_R|^{2}+\psi^{2}\tau^{2}|\nabla \zeta_R|^{2}\big)\,dx\\
= & \int_{B_{R^2}}\big(\varphi^{2}|\nabla(\zeta_{R}\sigma)|^2+\psi^{2}|\nabla(\zeta_{R}\tau)|^2\big)\,dx
-\int_{B_{R^2}}\big(\xi^{2}|\nabla\zeta_R|^{2}+\eta^{2}|\nabla \zeta_R|^{2}\big)\,dx.
\end{aligned}
\end{equation*}
Since  $\zeta_{R}\equiv1$ in $B_{R}$, the inequality above implies that
\begin{align*}  \int_{B_{R}}\big(\varphi^{2}|\nabla\sigma|^2+\psi^{2}|\nabla\tau|^2\big)\,dx\leq&\int_{B_{R^2}}\big(\varphi^{2}|\nabla(\zeta_{R}\sigma)|^2+\psi^{2}|\nabla(\zeta_{R}\tau)|^2\big)\,dx\\
    \leq&\int_{B_{R^2}}\big(\xi^{2}|\nabla\zeta_R|^{2}+\eta^{2}|\nabla \zeta_R|^{2}\big)\,dx.
\end{align*}
Employing \eqref{eq:g1} and \eqref{eq:g2}, we get
\begin{equation*}
    \int_{B_{R}}\big(\varphi^{2}|\nabla\sigma|^2+\psi^{2}|\nabla\tau|^2\big)\,dx\leq\frac{C}{\ln{R}},\quad\mbox{for }R>2.
\end{equation*}
It follows that the right-hand side tends to zero as $R\to \infty$, and hence 
$$
\int_{\mathbb{R}^3}\big(\varphi^{2} |\nabla\sigma|^2+\psi^{2} |\nabla\tau|^2\big)\,dx=0,
$$
which implies that $\sigma$ and $\tau$ are constant. Moreover, from \eqref{eq1} we conclude that
$$
\sigma=\tau=\mbox{constant}
$$
since  $\mathcal{P}(x)$  is strictly positive. 
This completes the proof of Lemma \ref{lem. Liouville-type result with cut-off function}.
\end{proof}

As a consequence, Theorem~\ref{thm. main theorem} becomes immediate in dimension $2$ instead of $3$. More precisely, we have the following corollary.
\begin{corollary}\label{cor. n=2 is more obvious}
    Assume that $(u,v)$ is a solution to \eqref{eq. simplified B-E} in $\mathbb{R}^{2}$, and that there exists a pair $(\varphi,\psi)$, such that $\varphi>0>\psi$ everywhere and $L\begin{bmatrix}\varphi\\\psi\end{bmatrix}$ always belongs to $\overline{\mathbb{R}_{+}}\times\overline{\mathbb{R}_{-}}$. 
Then one of the following holds:
\begin{itemize}
    \item $(u,v)$ is constant.
    \item Or, there exists a unit vector $\vec{\nu}\in\mathbb{R}^{2}$ such that
    \[
    (u(x),v(x)) = (u(x\cdot\vec{\nu}),v(x\cdot\vec{\nu})) \quad \text{with} \quad \partial_{\vec{\nu}} u > 0 > \partial_{\vec{\nu}} v.
    \]
 \end{itemize}   
\end{corollary}

\begin{proof}
    For $i=1,2$, denote
    \begin{equation*}
        (\xi_{i},\eta_{i}) := \Big(\frac{\partial u}{\partial x_{i}}, \frac{\partial v}{\partial x_{i}}\Big), \quad 
        (\sigma_{i},\tau_{i}) := \Big(\frac{\xi_{i}}{\varphi}, \frac{\eta_{i}}{\psi}\Big).
    \end{equation*}
    Let $Q_{(\varphi,\psi)}\begin{bmatrix}\sigma_{i}\\\tau_{i}\end{bmatrix}$ be defined as in \eqref{eq. ratio quadratic operator} for $i=1,2$. We first verify that the quadratic energy growth condition \eqref{eq. quadratic energy growth condition} holds. Indeed, it suffices to show
    \begin{equation*}
        \int_{B_{R}} \frac{|\nabla u|^{2} + |\nabla v|^{2}}{2} \, dx \leq C R^2.
    \end{equation*}
    This follows from the Schauder estimate
    \[
        \|\nabla u\|_{L^{\infty}(\mathbb{R}^{2})} + \|\nabla v\|_{L^{\infty}(\mathbb{R}^{2})} \leq C,
    \]
    together with the fact that the area of a ball in $\mathbb{R}^2$ grows quadratically with its radius.

    By Lemma~\ref{lem. energy growth}, there exist constants $c_1, c_2$ such that
    \begin{equation*}
        \sigma_{i} \equiv \tau_{i} \equiv c_{i}, \quad i=1,2.
    \end{equation*}
    By the definition of $(\sigma_{i},\tau_{i})$, we have
    \begin{equation}\label{eq. gradient of u and v in 2 dim}
        \nabla u(x) = \varphi(x) (c_1 \vec{e}_1 + c_2 \vec{e}_2), \quad 
        \nabla v(x) = \psi(x) (c_1 \vec{e}_1 + c_2 \vec{e}_2).
    \end{equation}

    We now distinguish two cases:
    \begin{itemize}
        \item \textbf{Case 1:} If $c_1 = c_2 = 0$, then $u$ and $v$ are both constant.
        \item \textbf{Case 2:} Otherwise, set
        \begin{equation*}
            \vec{\nu} = \frac{c_1 \vec{e}_1 + c_2 \vec{e}_2}{\sqrt{c_1^2 + c_2^2}}, \quad
            \vec{\nu}^{\perp} = \frac{-c_2 \vec{e}_1 + c_1 \vec{e}_2}{\sqrt{c_1^2 + c_2^2}}.
        \end{equation*}
        Then, from \eqref{eq. gradient of u and v in 2 dim}, we see that
        \[
            \vec{\nu}^{\perp} \cdot \nabla u(x) = \vec{\nu}^{\perp} \cdot \nabla v(x) = 0.
        \]
        Therefore, the functions depend only on $x \cdot \vec{\nu}$:
        \[
            (u(x), v(x)) = (u(x \cdot \vec{\nu}), v(x \cdot \vec{\nu})),
        \]
        with $u$ strictly increasing and $v$ strictly decreasing along the direction $\vec{\nu}$.
    \end{itemize}

    This completes the proof of Corollary~\ref{cor. n=2 is more obvious}.
\end{proof}

\section{Growth rate of the Ginzburg-Landau energy}
In this section, the key result is the following growth estimate for the Ginzburg-Landau energy.

\begin{lemma}\label{lem. energy growth}
    Under the assumptions of Theorem~\ref{thm. main theorem}, there exists a constant $C>0$ such that
    \[
        J(u,v,B_{R}) \leq C R^{n-1}
    \]
    for all sufficiently large $R$.
\end{lemma}

\subsection{Evolution of the energy}
The strategy to prove Lemma~\ref{lem. energy growth} is to translate the domain $B_R$ along the $\vec{e}_{n}$-direction to infinity. More precisely, we shall establish the following two statements:
\begin{itemize}
    \item[(1)] The energy variation during this translation process is at most $O(R^{n-1})$.
    \item[(2)] The limiting profile has energy bounded by $O(R^{n-1})$.
\end{itemize}

In this subsection, we focus on proving the first statement (see Lemma~\ref{lem. evolution of E(t)} below), while the second statement will be addressed in the next subsection.

\begin{lemma}\label{lem. evolution of E(t)}
    Assume that $(u,v)$ is a global solution to \eqref{eq. simplified B-E} with 
    \[
        \frac{\partial u}{\partial x_{n}} > 0 > \frac{\partial v}{\partial x_{n}},
    \]
    and let $R$ be sufficiently large. For each $t \in \mathbb{R}$, define the cube
    \[
        Q_{t} = [-R,R]^{n-1} \times [t-R, t+R], \qquad E(t) = J(u,v,Q_t).
    \]
    Then there exists a constant $C>0$ such that
    \[
        \int_{-\infty}^{\infty} \Big| \frac{dE}{dt} \Big| \, dt \le C R^{n-1}.
    \]
\end{lemma}

\begin{proof}
    Observe that translating the cube $Q_t$ is equivalent to translating the solution along the $\vec{e}_n$-direction. Define
    \[
        (u^{(t)}(x',x_n), v^{(t)}(x',x_n)) := (u(x', x_n+t), v(x', x_n+t)).
    \]
    Then $J(u,v,Q_t) = J(u^{(t)},v^{(t)},Q_0)$ and
    \[
        \frac{d}{dt}(u^{(t)},v^{(t)}) = \left( \frac{\partial u}{\partial x_n}(x', x_n+t), \frac{\partial v}{\partial x_n}(x', x_n+t) \right).
    \]

    By integration by parts,
    \begin{align*}
        \frac{dE}{dt} 
        &= \frac{d}{dt} \int_{Q_t} \left\{ \frac{|\nabla u|^2 + |\nabla v|^2}{2} + W(u,v) \right\} dx \\
        &= \int_{Q_t} \Big\{ \nabla u \cdot \nabla \frac{\partial u}{\partial x_n} + \nabla v \cdot \nabla \frac{\partial v}{\partial x_n} + W_u \frac{\partial u}{\partial x_n} + W_v \frac{\partial v}{\partial x_n} \Big\} dx \\
        &= \int_{Q_t} \Big\{ (W_u - \Delta u) \frac{\partial u}{\partial x_n} + (W_v - \Delta v) \frac{\partial v}{\partial x_n} \Big\} dx \\
        &\quad + \int_{\partial Q_t} \Big\{ \frac{\partial u}{\partial x_n} \frac{\partial u}{\partial \vec{\nu}} + \frac{\partial v}{\partial x_n} \frac{\partial v}{\partial \vec{\nu}} \Big\} d\mathcal{H}^{n-1}_{\partial Q_t} \\
        &= \int_{\partial Q_t} \Big\{ \frac{\partial u}{\partial x_n} \frac{\partial u}{\partial \vec{\nu}} + \frac{\partial v}{\partial x_n} \frac{\partial v}{\partial \vec{\nu}} \Big\} d\mathcal{H}^{n-1}_{\partial Q_t},
    \end{align*}
    where $\vec{\nu}$ denotes the outward unit normal on $\partial Q_t$. By Schauder estimates, $\frac{\partial u}{\partial \vec{\nu}}$ and $\frac{\partial v}{\partial \vec{\nu}}$ are uniformly bounded, so
    \begin{equation}\label{eq. dE/dt integral, s=1}
        \int_{-\infty}^{\infty}\Big|\frac{dE}{dt}\Big|dt\leq C\int_{-\infty}^{\infty}\int_{\partial Q_{t}}\Big\{\Big|\frac{\partial u}{\partial x_{n}}\Big|+\Big|\frac{\partial v}{\partial x_{n}}\Big|\Big\}d\mathcal{H}^{n-1}_{\partial Q_{t}}dt.
    \end{equation}

    To estimate the right-hand side, consider an arbitrary $x = (x', x_n) \in \mathbb{R}^n$:
    \begin{itemize}
        \item If its projection on $\mathbb{R}^{n-1}$, i.e.: $x'$, is outside $[-R,R]^{n}$, then $x$ does not contribute to the integral \eqref{eq. dE/dt integral, s=1}.
        \item If $x'\in(-R,R)^{n}$, then $x$ contributes twice to the integral \eqref{eq. dE/dt integral, s=1}, precisely: when $t=x_{n}\pm R$.
        \item If $x'\in\partial[-R,R]^{n}$, then $x$ contributes to the integral \eqref{eq. dE/dt integral, s=1} during a period of time, i.e.: $t\in[x_{n}-R,x_{n}+R]$.
    \end{itemize}

    Since $\mathcal{H}^{n-1}([-R,R]^{n-1})$ and $R \, \mathcal{H}^{n-2}(\partial[-R,R]^{n-1})$ are both of order $R^{n-1}$, we deduce
    \begin{align*}
        \int_{-\infty}^{\infty} \Big| \frac{dE}{dt} \Big| dt 
        &\le C R^{n-1} \sup_{x' \in [-R,R]^{n-1}} \int_{-\infty}^{\infty} \Big( \Big| \frac{\partial u}{\partial x_n}(x',x_n) \Big| + \Big| \frac{\partial v}{\partial x_n}(x',x_n) \Big| \Big) dx_n \\
        &\le 2 C R^{n-1} (b-a) \le 2 C R^{n-1},
    \end{align*}
    where we have used the monotonicity and boundedness of $u$ and $v$ in the $x_n$-direction. This completes the proof of Lemma~\ref{lem. evolution of E(t)}.
\end{proof}

\subsection{Limiting profile}
In this subsection, we analyze the limiting profile obtained by pushing the solution along the $\vec{e}_{n}$-direction to infinity. Since $(u,v)$ are monotone in the $\vec{e}_{n}$-direction, the limiting profile is well-defined pointwise, as described below.

\begin{definition}\label{def. limiting profile}
    We define the upper limiting profiles $\overline{u}, \overline{v} : \mathbb{R}^{n-1} \to \mathbb{R}$ by
    \[
        \overline{u}(x') := \lim_{x_n \to +\infty} u(x', x_n), \qquad 
        \overline{v}(x') := \lim_{x_n \to +\infty} v(x', x_n).
    \]
    Similarly, we define the lower limiting profiles $\underline{u}, \underline{v} : \mathbb{R}^{n-1} \to \mathbb{R}$ by
    \[
        \underline{u}(x') := \lim_{x_n \to -\infty} u(x', x_n), \qquad 
        \underline{v}(x') := \lim_{x_n \to -\infty} v(x', x_n).
    \]
\end{definition}

\begin{lemma}
    The limiting profiles $(\overline{u},\overline{v})$ and $(\underline{u},\underline{v})$ satisfy the following properties:
    \begin{itemize}
        \item[(1)] Both $(\overline{u},\overline{v})$ and $(\underline{u},\underline{v})$ have globally bounded derivatives of all orders. In particular, each of $\overline{u}, \overline{v}, \underline{u}, \underline{v}$ takes values in the interval $[a, b]$.
        \item[(2)] Both pairs $(\overline{u},\overline{v})$ and $(\underline{u},\underline{v})$ satisfy the system \eqref{eq. simplified B-E} in $\mathbb{R}^{\,n-1}$.
    \end{itemize}
\end{lemma}

\begin{proof}
    Since $u$ and $v$ are monotone in the $\vec{e}_{n}$ direction and by the range estimate in \cite[Proposition 1.2]{AFN21}, the limiting profiles $(\overline{u},\overline{v})$ and $(\underline{u},\underline{v})$ are well-defined and globally bounded between $a$ and $b$. Moreover, because all derivatives of $(u,v)$ are globally bounded in $\mathbb{R}^{n}$, the limits inherit these derivative bounds. 

    Indeed, consider the translated pair
    \begin{equation}\label{eq. translated pair}
        \big(u^{(t)}(x),v^{(t)}(x)\big):=\big(u(x',x_{n}+t),v(x',x_{n}+t)\big).
    \end{equation}
    By the Arzel\`{a}-Ascoli theorem, as $t \to \pm \infty$, this sequence converges to $(\overline{u},\overline{v})$ and $(\underline{u},\underline{v})$ (viewed as canonical extensions from $\mathbb{R}^{\,n-1}$ to $\mathbb{R}^{n}$) strongly in all $C^{k,\alpha}(\mathbb{R}^{n})$ spaces. 

    Finally, since each translated pair $(u^{(t)},v^{(t)})$ satisfies the system \eqref{eq. simplified B-E}, the limiting profiles $(\overline{u},\overline{v})$ and $(\underline{u},\underline{v})$ also satisfy \eqref{eq. simplified B-E} in $\mathbb{R}^{\,n-1}$.
\end{proof}

Next, we aim to show that the limiting profiles $(\overline{u},\overline{v})$ and $(\underline{u},\underline{v})$ are one-dimensional. To this end, it suffices to show the existence of functions $\varphi>0>\psi$ such that
\[
L \begin{bmatrix} \varphi \\ \psi \end{bmatrix} \in \overline{\mathbb{R}_{+}} \times \overline{\mathbb{R}_{-}},
\]
see Corollary~\ref{cor. n=2 is more obvious}.

\begin{lemma}\label{lem. classification of one limiting profile}
Let $n=3$. For the limiting profile $(\overline{u},\overline{v})$ given in Definition~\ref{def. limiting profile}, there exist two functions $\varphi>0>\psi$ defined on $\mathbb{R}^{2}$ such that
\begin{equation*}
\begin{cases}
-\Delta \varphi + W_{uu}(\overline{u},\overline{v})\,\varphi + W_{uv}(\overline{u},\overline{v})\,\psi \ge 0,\\[1ex]
-\Delta \psi + W_{vu}(\overline{u},\overline{v})\,\varphi + W_{vv}(\overline{u},\overline{v})\,\psi \le 0.
\end{cases}
\end{equation*}

Moreover, $(\overline{u},\overline{v})$ is one-dimensional. More precisely, one of the following holds:
\begin{itemize}
    \item $(\overline{u},\overline{v})$ is constant, taking one of the forms shown in Figure~\ref{fig. range}:
    \[
        (\overline{u},\overline{v}) \equiv (a,b), \quad (b,a), \quad \text{or} \quad \Bigl(\sqrt{\frac{1+\omega}{2+\alpha}}, \sqrt{\frac{1+\omega}{2+\alpha}}\Bigr).
    \]
    \item $(\overline{u},\overline{v})$ depends only on $x_{2}$ (up to a rotation in $\mathbb{R}^{2}$), with
    \[
        \frac{\partial \overline{u}}{\partial x_{2}} > 0 > \frac{\partial \overline{v}}{\partial x_{2}}.
    \]
\end{itemize}
\end{lemma}

\begin{proof}
    We consider a family of functionals related to the translated pair $(u^{(t)},v^{(t)})$ in \eqref{eq. translated pair}. We also denote
    \begin{equation*}
        W^{(t)}=W(u^{(t)},v^{(t)}).
    \end{equation*}
    Let $(\xi,\eta)$ be a pair of compactly supported smooth functions in $\mathbb{R}^{3}$, and let
    \begin{equation}\label{eq. E_t xi eta s=1}
        \mathcal{E}_{t}(\xi,\eta):=\int_{\mathbb{R}^{3}}\Big\{|\nabla\xi|^{2}+|\nabla\eta|^{2}+\begin{bmatrix}\xi\\\eta\end{bmatrix}^{T}D^{2}W^{(t)}\begin{bmatrix}\xi\\\eta\end{bmatrix}\Big\}dx.
    \end{equation}
    
    \textbf{Step 1: Positivity of $\mathcal{E}_{t}(\xi,\eta)$.} We prove that $\mathcal{E}_{t}(\xi,\eta)$ is always non-negative. Recall that $\displaystyle\frac{\partial u^{(t)}}{\partial x_{3}}>0>\frac{\partial v^{(t)}}{\partial x_{3}}$ by the monotonicity assumption, and that $\bigl(\frac{\partial u^{(t)}}{\partial x_{3}},\frac{\partial v^{(t)}}{\partial x_{3}}\bigr)$ is a solution to the Schr\"odinger system
    \begin{equation*}
        \left\{\begin{aligned}
            -\Delta\frac{\partial u^{(t)}}{\partial x_{3}}+W_{uu}^{(t)}\frac{\partial u^{(t)}}{\partial x_{3}}+W_{uv}^{(t)}\frac{\partial v^{(t)}}{\partial x_{3}}=0,\\
            -\Delta\frac{\partial v^{(t)}}{\partial x_{3}}+W_{vu}^{(t)}\frac{\partial u^{(t)}}{\partial x_{3}}+W_{vv}^{(t)}\frac{\partial v^{(t)}}{\partial x_{3}}=0.
        \end{aligned}\right.
    \end{equation*}
    as established in Lemma~\ref{lem. derivative satisfies the Schrodinger equation}.
    
        Set
    \begin{equation*}
        (\sigma,\tau)=\Big(\frac{\xi}{\partial_{x_{3}}u^{(t)}},\frac{\eta}{\partial_{x_{3}}v^{(t)}}\Big).
    \end{equation*}
    Since  $\displaystyle\frac{\partial u}{\partial x_{3}}>0>\frac{\partial v}{\partial x_{3}}$ and that $(\xi,\eta)$ is compactly supported, we have that $(\sigma,\tau)$ is smooth and compactly supported in $\mathbb{R}^{3}$.  Then, a straightforward computation gives
    \begin{align*}
        \mathcal{E}_{t}(\xi,\eta)=&\int_{\mathbb{R}^{3}}\Big\{\Big|\frac{\partial u^{(t)}}{\partial x_{3}}\nabla\sigma\Big|^{2}+\nabla\frac{\partial u^{(t)}}{\partial x_{3}}\cdot\nabla(\sigma^{2}\frac{\partial u^{(t)}}{\partial x_{3}})+\Big|\frac{\partial v^{(t)}}{\partial x_{3}}\nabla\tau\Big|^{2}+\nabla\frac{\partial v^{(t)}}{\partial x_{3}}\cdot\nabla(\tau^{2}\frac{\partial v^{(t)}}{\partial x_{3}})\Big\}dx\\
        &+\int_{\mathbb{R}^{3}}\Big\{W_{uu}^{(t)}\Big(\frac{\partial u^{(t)}}{\partial x_{3}}\sigma\Big)^{2}+W_{vv}^{(t)}\Big(\frac{\partial v^{(t)}}{\partial x_{3}}\tau\Big)^{2}+2W_{uv}^{(t)}\Big(\frac{\partial u^{(t)}}{\partial x_{3}}\sigma\Big)\Big(\frac{\partial v^{(t)}}{\partial x_{3}}\tau\Big)\Big\}dx\\
        =&\int_{\mathbb{R}^{3}}\Big\{\Big|\frac{\partial u^{(t)}}{\partial x_{3}}\nabla\sigma\Big|^{2}+\Big|\frac{\partial v^{(t)}}{\partial x_{3}}\nabla\tau\Big|^{2}\Big\}dx-\int_{\mathbb{R}^{3}}\Big\{\sigma^{2}\frac{\partial u^{(t)}}{\partial x_{3}}\Delta\frac{\partial u^{(t)}}{\partial x_{3}}+\tau^{2}\frac{\partial v^{(t)}}{\partial x_{3}}\Delta\frac{\partial v^{(t)}}{\partial x_{3}}\Big\}dx\\
        &+\int_{\mathbb{R}^{3}}\Big\{\Big(\sigma^{2}\frac{\partial u^{(t)}}{\partial x_{3}}\Big)W_{uu}^{(t)}\frac{\partial u^{(t)}}{\partial x_{3}}+\Big(\tau^{2}\frac{\partial v^{(t)}}{\partial x_{3}}\Big)W_{vv}^{(t)}\frac{\partial v^{(t)}}{\partial x_{3}}+2\sigma\tau W_{uv}^{(t)}\frac{\partial u^{(t)}}{\partial x_{3}}\frac{\partial v^{(t)}}{\partial x_{3}}\Big\}dx\\
        =&\int_{\mathbb{R}^{3}}\Big\{\Big|\frac{\partial u^{(t)}}{\partial x_{3}}\nabla\sigma\Big|^{2}-(\sigma-\tau)^{2}W_{uv}^{(t)}\frac{\partial u^{(t)}}{\partial x_{3}}\frac{\partial v^{(t)}}{\partial x_{3}}\Big\}dx\geq0.
    \end{align*}
    Here, we have again used the fact that $\displaystyle W_{uv}^{(t)}\frac{\partial u^{(t)}}{\partial x_{3}}\frac{\partial v^{(t)}}{\partial x_{3}}\leq0$, see the discussion in \eqref{eq. W_uv>0}.

\textbf{Step 2: Sending $t \to +\infty$.}  
We define $\mathcal{E}_{+\infty}(\xi,\eta)$ analogously to \eqref{eq. E_t xi eta s=1}, replacing $D^{2}W^{(t)}$ with
\begin{equation*}
D^{2}W^{(+\infty)} := D^{2}W(u^{(+\infty)},v^{(+\infty)}),
\end{equation*}
where 
\[
u^{(+\infty)}(x',x_{3}) := \overline{u}(x'), \qquad v^{(+\infty)}(x',x_{3}) := \overline{v}(x').
\]  
Since we have already shown that $\mathcal{E}_{t}(\xi,\eta) \ge 0$, passing to the limit yields
\[
\mathcal{E}_{+\infty}(\xi,\eta) \ge 0
\]  
for every compactly supported smooth pair $(\xi,\eta)$.  

Next, we set
\begin{equation*}
D^{2}\overline{W} := D^{2}W(\overline{u},\overline{v}),
\end{equation*}
which is defined on $\mathbb{R}^{n-1} = \mathbb{R}^{2}$. Let $(\xi,\eta)$ be a pair of compactly supported functions in $\mathbb{R}^{2}$, and define
\begin{equation*}
\overline{\mathcal{E}}(\xi,\eta) := \int_{\mathbb{R}^{2}} \Biggl\{ |\nabla \xi|^{2} + |\nabla \eta|^{2} + 
\begin{bmatrix} \xi \\ \eta \end{bmatrix}^{T} D^{2} \overline{W} \begin{bmatrix} \xi \\ \eta \end{bmatrix} \Biggr\} dx.
\end{equation*}

\textbf{Step 3: Positivity of $\overline{\mathcal{E}}(\xi,\eta)$.}  
We now show that $\overline{\mathcal{E}}(\xi,\eta) \ge 0$ for any compactly supported pair $(\xi,\eta)$ in $\mathbb{R}^{2}$.  

Suppose, on the contrary, that $\overline{\mathcal{E}}(\xi,\eta) < 0$ for some $(\xi,\eta)$. We then denote
\begin{equation}\label{eq. -T and diameter}
\overline{\mathcal{E}}(\xi,\eta) = -\mathcal{T} < 0, \quad 
\mathcal{S} := \operatorname{supp}(\xi) \cup \operatorname{supp}(\eta) \subset\subset \mathbb{R}^{2}, \quad 
\rho := \sup_{x' \in \mathcal{S}} |x'|.
\end{equation}

Let $0 \le Z_{R}(x_{3}) \le 1$ be a smooth cut-off function such that
\[
Z_{R}(x_{3}) \equiv 1 \text{ for } |x_{3}| \le R, \quad
Z_{R}(x_{3}) \equiv 0 \text{ for } |x_{3}| \ge R+1, \quad
|\nabla Z_{R}| \le C.
\]  
Define
\[
(\xi_{R}(x),\eta_{R}(x)) := (\xi(x') Z_{R}(x_{3}), \eta(x') Z_{R}(x_{3})),
\]  
which is a compactly supported pair in $\mathbb{R}^{3}$. Then, by Step 2, we have
\[
\mathcal{E}_{+\infty}(\xi_{R},\eta_{R}) \ge 0 \quad \text{for all } R \ge 1.
\]

For $R$ sufficiently large, we can decompose
\begin{align*}
\mathcal{E}_{+\infty}(\xi_{R},\eta_{R}) 
&= \int_{-R}^{R} \int_{\mathbb{R}^{2}} \Biggl\{ |\nabla \xi|^{2} + |\nabla \eta|^{2} + 
\begin{bmatrix}\xi\\\eta\end{bmatrix}^{T} D^{2}\overline{W} \begin{bmatrix}\xi\\\eta\end{bmatrix} \Biggr\} dx' dx_{3} \\
&\quad + \Biggl\{ \int_{-R-1}^{-R} + \int_{R}^{R+1} \Biggr\} \int_{\mathbb{R}^{2}} 
\Biggl\{ |\nabla \xi_{R}|^{2} + |\nabla \eta_{R}|^{2} + 
\begin{bmatrix}\xi_{R}\\\eta_{R}\end{bmatrix}^{T} D^{2}\overline{W} \begin{bmatrix}\xi_{R}\\\eta_{R}\end{bmatrix} \Biggr\} dx' dx_{3} \\
&\le -2 R \mathcal{T} + C(\xi,\eta) < 0,
\end{align*}
which is a contradiction. Here, $C(\xi,\eta)$ is a constant depending only on 
\(\|\xi\|_{C^{0,1}(\mathbb{R}^{2})} + \|\eta\|_{C^{0,1}(\mathbb{R}^{2})}\) 
and the radius \(\rho\) defined in \eqref{eq. -T and diameter}.

    \textbf{Step 4: Construction of $(\varphi,\psi)$.} In this step, we intend to consruct two entire functions $\varphi>0>\psi$ defined on $\mathbb{R}^{2}$, such that the following Schr\"odinger inequalities holds:
    \begin{equation}\label{eq. varphi and psi Schrodinger inequality}
        \left\{\begin{aligned}
            &-\Delta\varphi+\overline{W}_{uu}\varphi+\overline{W}_{uv}\psi\geq0,\\
            &-\Delta\psi+\overline{W}_{vu}\varphi+\overline{W}_{vv}\psi\leq0.
        \end{aligned}\right.
    \end{equation}
    
    As we have established that $\overline{\mathcal{E}}$ is positive semi-definite for all compactly supported pairs $(\xi,\eta)$, we can consider the following eigenvalue problem. For every sufficiently large radius $R$, we intend to find the minimizers of the following functional (the value must be non-negative):
    \begin{equation}\label{eq. eigenvalue functional}
        \frac{\overline{\mathcal{E}}(\xi,\eta)}{\int_{\mathbb{R}^{2}}(|\xi|^{2}+|\eta|^{2})dx},\quad\mbox{where }(\xi,\eta)\in C^{\infty}_{0}(B_{R}).
    \end{equation}
    Let $\lambda_{R}$ be the infimum of such a functional. Then $\lambda_{R}\geq0$ and is non-increasing in $R$. In particular, $\lambda_{R}$ is uniformly bounded for $R\geq100$. Moreover, there exists an eigen-pair $(\varphi_{R},\psi_{R})\in C^{\infty}_{0}(B_{R})$, such that
    \begin{equation*}
        \frac{\overline{\mathcal{E}}(\varphi_{R},\psi_{R})}{\int_{\mathbb{R}^{2}}(|\varphi_{R}|^{2}+|\psi_{R}|^{2})dx}=\inf_{(\xi,\eta)\in C^{\infty}_{0}(B_{R})}\frac{\overline{\mathcal{E}}(\xi,\eta)}{\int_{\mathbb{R}^{2}}(|\xi|^{2}+|\eta|^{2})dx}=\lambda_{R}.
    \end{equation*}
    Notice that in the expression of $\overline{\mathcal{E}}(\varphi_{R},\psi_{R})$, we have $\overline{W}_{uv}>0$. It is then beneficial to replace $(\varphi_{R},\psi_{R})$ with $(|\varphi_{R}|,-|\psi_{R}|)$, which does not increase the functional \eqref{eq. eigenvalue functional}. Therefore, without loss of generality, we assume that $(\varphi_{R},\psi_{R})$ is chosen such that
    \begin{equation*}
        \varphi_{R}\geq0\geq\psi_{R}\mbox{ in }B_{R},\quad\varphi_{R}=\psi_{R}=0\mbox{ in }B_{R}^{c}.
    \end{equation*}
    The Euler-Lagrange equation satisfied by the minimizing pair $(\varphi_{R},\psi_{R})$ in $B_{R}$ are
    \begin{equation}\label{eq. varphi_R and psi_R equation}
        \left\{\begin{aligned}
            &-\Delta\varphi_{R}+\overline{W}_{uu}\varphi_{R}+\overline{W}_{uv}\psi_{R}=\lambda_{R}\varphi_{R},\\
            &-\Delta\psi_{R}+\overline{W}_{vu}\varphi_{R}+\overline{W}_{vv}\psi_{R}=\lambda_{R}\psi_{R}.
        \end{aligned}\right.
    \end{equation}
    
    We observe that the difference
    \begin{equation*}
        \mathcal{D}_{R}=\varphi_{R}-\psi_{R}
    \end{equation*}
    is non-negative in $\mathbb{R}^{2}$  and does not vanish identically in $B_R.$
  Moreover, by subtracting the two equations in \eqref{eq. varphi_R and psi_R equation}, we see that $\mathcal{D}_{R}$ satisfies the following inequality in $B_{R}$:
    \begin{equation*}
        -C\mathcal{D}_{R}\leq\Delta\mathcal{D}_{R}\leq C\mathcal{D}_{R},\quad C:=100\max_{R\geq100}\lambda_{R}+100\|D^{2}\overline{W}\|_{L^{\infty}(\mathbb{R}^{2})}.
    \end{equation*}
    By the strong maximum principle, it follows that $\mathcal{D}_{R}$ must be strictly positive in $B_{R}$. This allows us to normalize the pair $(\varphi_{R},\psi_{R})$ such that 
    $$\mathcal{D}_{R}(0)=\varphi_{R}(0)-\psi_{R}(0)=1.$$  
    We apply the Harnack principle to $\mathcal{D}_{R}$, and then employ the Schauder estimate to $\varphi_{R}$ and $\psi_{R}$, and have the following universal estimate depending minimally on $R$.  Specifically, there exists a constant $C_{\rho}$ depending only on $\rho$, such that for every $\rho\geq100$ and $R\geq2\rho$, we have
 \begin{itemize}
        \item[(1)] $C_{\rho}^{-1}\leq\mathcal{D}_{R}(x)\leq C_{\rho}$ for all $x\in B_{\rho}$.
        \item[(2)] $\|\varphi_{R}\|_{C^{2,\alpha}(B_{\rho})}\leq C_{\rho}$ and $\|\psi_{R}\|_{C^{2,\alpha}(B_{\rho})}\leq C_{\rho}$.
    \end{itemize}
    In fact, we have used the assumption $\varphi_{\rho}\geq0\geq\psi_{\rho}$ in (2), which allows us to control $\|\varphi_{R}\|_{L^{\infty}(B_{\rho})}$ and $\|\psi_{R}\|_{L^{\infty}(B_{\rho})}$ using $\|\mathcal{D}_{R}\|_{L^{\infty}(B_{\rho})}$.

   We can then pass the limit as  $R\to\infty$. By a diagonal selection argument, we see that there exists a subsequence $\rho_{k}\to\infty$, such that
    \begin{equation*}
        (\varphi_{\rho_{k}},\psi_{\rho_{k}})\to(\varphi,\psi)\mbox{ in the }C^{2,\alpha}(B_{R})\mbox{ sense, for all }R\geq100.
    \end{equation*}
    We have that $\mathcal{D}_{\infty}=\varphi-\psi$ is nowhere zero, since $\mathcal{D}_{R}\geq C_{\rho}^{-1}$ in each $B_{\rho}$ for large $R$. Moreover, $\varphi\geq0\geq\psi$, and the pair $(\varphi,\psi)$ satisfies the limiting equation of \eqref{eq. varphi_R and psi_R equation}, which reads as:
    \begin{equation*}
        \left\{\begin{aligned}
            &-\Delta\varphi+\overline{W}_{uu}\varphi+\overline{W}_{uv}\psi=\lambda_{\infty}\varphi\geq0,\\
            &-\Delta\psi+\overline{W}_{vu}\varphi+\overline{W}_{vv}\psi=\lambda_{\infty}\psi\leq0,
        \end{aligned}\right.\quad\mbox{where }\lambda_{\infty}:=\lim_{R\to\infty}\lambda_{R}\geq0.
    \end{equation*}
    This verifies the inequality \eqref{eq. varphi and psi Schrodinger inequality}.
    
    Finally, it remains to show that $\varphi$ and $\psi$ are both nowhere zero. To verify this, we argue by contradiction. Without loss of generality, suppose that $\varphi(x'_{*})=0$ for some $x'_{*}\in\mathbb{R}^{2}$. By the positivity of $\overline{W}_{uv}$, the negativity of $\psi$, and the global boundedness of $\overline{W}_{uu}$, we have
    \begin{equation*}
        \Delta\varphi\leq\overline{W}_{uu}\varphi\leq\|\overline{W}_{uu}\|_{L^{\infty}(\mathbb{R}^{2})}\cdot\varphi.
    \end{equation*}
    If we assume $\varphi(x'_{*})=0$ for some $x'_{*}\in\mathbb{R}^{2}$, then it follows from the strong maximum principle that $\varphi\equiv0$ in $\mathbb{R}^{2}$. As a result, $\psi(x')\equiv-\mathcal{D}_{\infty}(x')>0$ everywhere in $\mathbb{R}^{2}$, and thus
    \begin{equation*}
        \Delta\varphi\leq\overline{W}_{uv}\psi<0\mbox{ everywhere in }\mathbb{R}^{2}.
    \end{equation*}
    This violates the previously obtained consequence $\varphi\equiv0$, so we have reached a contradiction, meaning that we have shown that $\varphi>0$ in $\mathbb{R}^{2}$. Similarly, we can show that $\psi<0$ in $\mathbb{R}^{2}$.

    \textbf{Step 5: Classification of the limiting profile.} We now describe the properties of the limiting profile $(\overline{u},\overline{v})$. Using the pair $(\varphi,\psi)$, we can apply Corollary~\ref{cor. n=2 is more obvious} to $(\overline{u},\overline{v})$. It follows  that either $(\overline{u},\overline{v})$ is constant, or $(\overline{u},\overline{v})$ depends only on $x_{2}$ (up to a rotation in $\mathbb{R}^{2}$). In the first case, the only constant positive solutions are given by (see Figure~\ref{fig. range}):
    \begin{equation*}
        (\overline{u},\overline{v})\equiv(a,b),\mbox{ or }(b,a),\mbox{ or }(\sqrt{\frac{1+\omega}{2+\alpha}},\sqrt{\frac{1+\omega}{2+\alpha}}).
    \end{equation*}
    In the second case, the monotonicity in the $x_{2}$ direction also follows directly from Corollary~\ref{cor. n=2 is more obvious}.
\end{proof}
\begin{remark}
    In Step 4 of the proof of Lemma~\ref{lem. classification of one limiting profile}, one can also refer to Chen-Zhao \cite{CZ97} for the Harnack principle for coupled systems.
\end{remark}
It is natural to expect that  Lemma~\ref{lem. classification of one limiting profile} also applies to  the other limiting profile $(\underline{u},\underline{v})$. Consequently, we obtain the following description of the behavior of the two limiting profiles.
\begin{lemma}\label{lem. classification of two limiting profiles}
    Let $(\overline{u},\overline{v})$ and $(\underline{u},\underline{v})$ be the two limiting profiles as defined in Definition~\ref{def. limiting profile}. Assume that  
    $$\displaystyle\frac{\partial u}{\partial x_{3}}>0>\frac{\partial v}{\partial x_{3}}$$ in $\mathbb{R}^{3}$. Tthen at least one of the following  cases occurs:
    \begin{itemize}
        \item Case 1: $(\overline{u},\overline{v})\equiv(b,a)$.
        \item Case 2: $(\underline{u},\underline{v})\equiv(a,b)$.
        \item Case 3: Up to a rotation, $(\overline{u},\overline{v})$ depends only on $x_{2}$ and $\displaystyle\frac{\partial\overline{u}}{\partial x_{2}}>0>\frac{\partial\overline{v}}{\partial x_{2}}$.
        \item Case 4: Up to a rotation, $(\underline{u},\underline{v})$ depends only on $x_{2}$ and $\displaystyle\frac{\partial\underline{u}}{\partial x_{2}}>0>\frac{\partial\underline{v}}{\partial x_{2}}$.
    \end{itemize}
    In Cases 1 and 3, we have that the Ginzburg-Landau energy of the pair $(\overline{u},\overline{v})$ in $\mathbb{R}^{2}$, as defined in \eqref{eq. Ginzburg-Landau energy, s=1}, satisfies the growth estimate 
    \begin{equation}\label{eq. limiting profile GL energy growth}
        J(\overline{u},\overline{v},B_{R}')\leq CR,\quad\mbox{for all }R\geq100.
    \end{equation}
    where  $B_{R}'$ represents the disc in $\mathbb{R}^{2}$ with radius $R$.
    
    In Cases 2 and  4, we have a similar estimate for the pair $(\underline{u},\underline{v})$, i.e.:
    \begin{equation*}
        J(\underline{u},\underline{v},B_{R}')\leq CR,\quad\mbox{for all }R\geq100.
    \end{equation*}
\end{lemma}
\begin{proof}
   \textbf{Step 1: Classification of the four cases.}  
We first assume that Cases 3 and 4 do not occur, and aim to show that at least one of Cases 1 and 2 must hold.  

In fact, by Lemma~\ref{lem. classification of one limiting profile}, both $(\overline{u},\overline{v})$ and $(\underline{u},\underline{v})$ are constant solutions.  
Suppose, in addition, that Case 1 does not hold. Then we have
\begin{equation*}
(\overline{u},\overline{v}) \equiv \left(\sqrt{\frac{1+\omega}{2+\alpha}}, \sqrt{\frac{1+\omega}{2+\alpha}}\right) 
\quad \text{or} \quad (\overline{u},\overline{v}) \equiv (a,b).
\end{equation*}

Since 
\[
\frac{\partial u}{\partial x_{3}} > 0 > \frac{\partial v}{\partial x_{3}},
\] 
it follows that $\underline{u} < \overline{u}$ and $\underline{v} > \overline{v}$. In particular,
\begin{equation*}
\underline{u} < \max \left\{ \sqrt{\frac{1+\omega}{2+\alpha}}, a \right\} = \sqrt{\frac{1+\omega}{2+\alpha}}, 
\quad 
\overline{v} > \min \left\{ \sqrt{\frac{1+\omega}{2+\alpha}}, b \right\} = \sqrt{\frac{1+\omega}{2+\alpha}}.
\end{equation*}

Hence, the only constant solution remaining in Lemma~\ref{lem. classification of one limiting profile} for $(\underline{u},\underline{v})$ is $(a,b)$.  
This completes the classification of the limiting profiles.

\textbf{Step 2: Energy estimate.}  
Next, we establish the growth estimate \eqref{eq. limiting profile GL energy growth} for the Ginzburg-Landau energy.  

The estimates in Cases 1 and 2 are trivial, since both the infinitesimal Dirichlet energy and the potential energy vanish identically. The more delicate cases are Cases 3 and 4. We will provide a detailed proof of \eqref{eq. limiting profile GL energy growth} in Case 3; Case 4 follows by the same argument.  

In Case 3, Lemma~\ref{lem. classification of one limiting profile} implies that
\[
(\overline{u}(x'),\overline{v}(x')) = (U(x_{2}), V(x_{2})),
\] 
where $(U(t), V(t))$ is a one-dimensional solution satisfying the monotonicity
\[
U'(t) > 0 > V'(t).
\]  

Observe that the width of $B_{R}'$ in the $x_{1}$-direction is of order $O(R)$.  
Hence, to prove the energy inequality \eqref{eq. limiting profile GL energy growth}, it suffices to show that the one-dimensional pair $(U,V)$ satisfies the uniform energy bound
\[
J(U,V,[-R,R]) \le C.
\]  
More precisely, we need to verify that
\begin{equation}\label{eq. limiting profile, 1d, energy bound}
\int_{-R}^{R} \Big\{ \frac{|U'(t)|^{2} + |V'(t)|^{2}}{2} + W(U(t),V(t)) \Big\} dt \le C.
\end{equation}

The idea is to apply Lemma~\ref{lem. evolution of E(t)}.  
For the one-dimensional pair $(U,V)$ with monotonicity 
\[
U'(t) > 0 > V'(t),
\] 
we define the limiting profiles $(\overline{U},\overline{V})$ and $(\underline{U},\underline{V})$ in the same manner as in Definition~\ref{def. limiting profile}.  

Since these limiting profiles are defined on $\mathbb{R}^{1-1} = \mathbb{R}^{0} = \{0\}$, they must be constant. Moreover, they satisfy the system \eqref{eq. simplified B-E} in $\mathbb{R}^{0}$.  
Hence, the only possible values of the pairs $(\overline{U},\overline{V})$ and $(\underline{U},\underline{V})$ are
\[
(\overline{U},\overline{V}),\, (\underline{U},\underline{V}) \in \Big\{ (a,b),\, (b,a),\, \left(\sqrt{\frac{1+\omega}{2+\alpha}}, \sqrt{\frac{1+\omega}{2+\alpha}}\right) \Big\}.
\]  

Analogously to the discussion in Step 1, at least one of the following two cases must occur:
\begin{itemize}
    \item Case A: $(\overline{U},\overline{V}) = (b,a)$,
    \item Case B: $(\underline{U},\underline{V}) = (a,b)$.
\end{itemize}

Without loss of generality, assume that Case A holds. Let $H$ be sufficiently large and send $H \to +\infty$. Then
\[
\lim_{H \to +\infty} E(H) = \lim_{H \to +\infty} \int_{H-R}^{H+R} \Big\{ \frac{|U'(t)|^{2} + |V'(t)|^{2}}{2} + W(U(t),V(t)) \Big\} dt = 0.
\]

Applying the one-dimensional version of Lemma~\ref{lem. evolution of E(t)} to the pair $(U,V)$ then establishes \eqref{eq. limiting profile, 1d, energy bound}.  
This completes Step 2 of the proof of Lemma~\ref{lem. classification of two limiting profiles}.
\end{proof}

\subsection{Proof of Lemma~\ref{lem. energy growth} and Theorem~\ref{thm. main theorem}}
Finally, we prove Lemma~\ref{lem. energy growth} and Theorem~\ref{thm. main theorem}.

\begin{proof}[Proof of Lemma~\ref{lem. energy growth}]
By Lemma~\ref{lem. classification of two limiting profiles}, we may assume without loss of generality that either Case 1 or Case 3 holds. Consequently, the estimate \eqref{eq. limiting profile GL energy growth} is valid.  

Applying Lemma~\ref{lem. evolution of E(t)} in dimension $n=3$, we obtain
\begin{equation}\label{eq. proof of main, energy bound}
J(u,v,[-R,R]^{3}) 
\leq \int_{-\infty}^{\infty} \left| \frac{dE}{dt} \right| dt + R \cdot J(\overline{u},\overline{v},[-R,R]^{2}) 
\leq C R^{2}.
\end{equation}

For sufficiently large $R$, the same estimate holds when the cube $[-R,R]^{3}$ is replaced by the ball $B_{R}$. This completes the proof of Lemma~\ref{lem. energy growth}.
\end{proof}

\begin{proof}[Proof of Theorem~\ref{thm. main theorem}]
We argue similarly as in Corollary~\ref{cor. n=2 is more obvious}.  
We set
\begin{equation*}
    (\varphi,\psi):=\left(\frac{\partial u}{\partial x_{3}},\frac{\partial v}{\partial x_{3}}\right), 
    \qquad
    (\xi_{i},\eta_{i}):=\left(\frac{\partial u}{\partial x_{i}},\frac{\partial v}{\partial x_{i}}\right), 
    \quad i=1,2.
\end{equation*}
By Lemma~\ref{lem. derivative satisfies the Schrodinger equation} together with Lemma~\ref{lem. negative definite form}, we deduce that for
\begin{equation*}
    (\sigma_{i},\tau_{i})
    :=\left(\frac{\xi_{i}}{\varphi},\frac{\eta_{i}}{\psi}\right),
\end{equation*}
there exists a strictly positive function $\mathcal{P}(x)$ such that
\begin{equation*}
    Q_{(\varphi,\psi)}
    \begin{bmatrix}
        \sigma_{i}\\
        \tau_{i}
    \end{bmatrix}
    \;\leq\; -(\sigma_{i}-\tau_{i})^{2}\,\mathcal{P}(x)
    \qquad \text{everywhere in }\mathbb{R}^{n}.
\end{equation*}

In view of the growth rate estimate for the Ginzburg-Landau energy established in Lemma~\ref{lem. energy growth}, we can apply Lemma~\ref{lem. Liouville-type result with cut-off function}. This yields that both $(\sigma_{1},\tau_{1})$ and $(\sigma_{2},\tau_{2})$ are constants, and moreover satisfy $\sigma_{i}=\tau_{i}$ for $i=1,2$.  

Finally, by repeating the argument used in Corollary~\ref{cor. n=2 is more obvious}, we conclude that the pair $(u,v)$ is one-dimensional. By rotating the space, we may assume that now $(u,v)$ depends only on $x_3$. Its limiting profiles $(\overline{u},\overline{v})$ and $(\underline{u},\underline{v})$ must then be constant solutions, and $(\overline{u},\overline{v}) = (b,a)$, $(\underline{u},\underline{v}) = (a,b)$. Here, we remark that the limiting profiles cannot be $(\sqrt{\frac{1+\omega}{2+\alpha}},\sqrt{\frac{1+\omega}{2+\alpha}})$, otherwise it violates Step 3 of the proof of Lemma~\ref{lem. classification of one limiting profile}.
\end{proof}

\vspace{2mm}
\noindent \textbf{Acknowledgments.}
Wu is partially supported by National Natural Science Foundation of China (Grant No. 12401133) and the Guangdong Basic and Applied Basic Research Foundation (2025B151502069).

\vspace{2mm}
\noindent \textbf{Conflict of interest.} The authors do not have any possible conflicts of interest.

\vspace{2mm}

\noindent \textbf{Data availability statement.}
 Data sharing is not applicable to this article, as no data sets were generated or analyzed during the current study.




\end{document}